\documentclass[12pt,reqno]{amsart}

\usepackage[top=1in, bottom=1in, left=1in, right=1in]{geometry}
\usepackage{times, amsthm, amssymb, amsmath, amsfonts, bm, graphicx,setspace}
\usepackage{mathrsfs,wasysym}
\usepackage[pagebackref=true,pdftex]{hyperref}
\usepackage[all]{xy}
\usepackage[usenames,dvipsnames]{color}

\newtheorem{theorem}{Theorem}[section]
\newtheorem{lemma}[theorem]{Lemma}
\newtheorem{corollary}[theorem]{Corollary}

\theoremstyle{definition}
\newtheorem{example}[theorem]{Example}

\newtheorem{definition}[theorem]{Definition}

\theoremstyle{plain}

        {\begin{list}
                {\noindent\makebox[0mm][r]{(\roman{enumi})}}
                {\leftmargin=5.5ex \usecounter{enumi}}
        }
        {\end{list}}

\DeclareMathAlphabet{\mathpzc}{OT1}{pzc}{m}{it}

\def\<{\langle}
\def\>{\rangle}

\def\CC{{\mathbb C}}

\def\RR{{\mathbb R}}

\def\ZZ{{\mathbb Z}}

\def\boldone{\boldsymbol{1}}
\def\boldzero{\boldsymbol{0}}

\def\del{\partial}

\def\dx#1{{\partial_{x_{#1}}}}

\def\dz#1{{\partial_{z_{#1}}}}

\def\Horn{{\rm Horn}}
\def\nHorn{{\rm nHorn}}
\def\sHorn{{\rm sHorn}}
\def\qdeg{{\rm qdeg}} 
\def\tdeg{{\rm tdeg}} 

\def\hom{\operatorname{Hom}}
\def\ext{\operatorname{Ext}}

\def\Var{{\rm Var}}

\def\gr{{\rm gr}}

\protect

\def\barX{{\bar{X}}}
\def\barZ{{\bar{Z}}}

\def\endrk{\hfill$\hexagon$}

\newcommand*{\defeq}{\mathrel{\vcenter{\baselineskip0.5ex \lineskiplimit0pt
                     \hbox{\scriptsize.}\hbox{\scriptsize.}}}%
                     =}

\numberwithin{equation}{section}
\begin{document}
\title[On normalized Horn Systems]{On normalized Horn systems}

\author[Berkesch]{Christine Berkesch}
\address{School of Mathematics, University of Minnesota, Minneapolis, MN 55455.}
\email{cberkesc@umn.edu}
\thanks{CBZ was partially supported by NSF Grants DMS 1440537 and DMS 1661962.
}

\author[Matusevich]{Laura Felicia Matusevich}
\address{Department of Mathematics \\
Texas A\&M University \\ College Station, TX 77843.}
\email{laura@math.tamu.edu}
\thanks{LFM was partially supported by NSF Grants DMS 0703866, DMS
  1001763, and a Sloan Research Fellowship.} 

\author[Walther]{Uli Walther}
\address{Department of Mathematics\\ Purdue University\\ West Lafayette, IN \ 47907.}
\email{walther@math.purdue.edu}
\thanks{UW was partially supported by NSF Grant DMS 0901123 and 
DMS 1401392.}

\subjclass[2010]{
Primary: 
14L30, 
33C70; 
Secondary: 
13N10, 
14M25, 
32C38. 
}

\begin{abstract}
  We characterize the (regular) holonomicity of Horn systems
  of differential equations under a hypothesis that
  captures the most widely studied classical hypergeometric systems.
\end{abstract}
\maketitle

\setcounter{tocdepth}{1}
\vspace{-4mm}

\section{Introduction}
\label{sec:intro}

Let $\barZ = \CC^m$ with coordinates $z_1,z_2,\dots, z_m$, and denote by 
$\dz{1},\dz{2},\dots,\dz{m}$ the partial derivative operators $\del/\del
z_1,\dots,\del/\del z_m$. The Weyl algebra $D_\barZ$, generated by
the $z_i$ and $\dz{i}$, is the ring of algebraic differential operators on
$\barZ$.

The goal of this article is to obtain $D$-module theoretic results
about \emph{normalized Horn systems}; in particular, we seek criterion for the
following two properties. A (left) $D_\barZ$-module $M$ is
\emph{holonomic} if $\ext_{D_\barZ}^j(M,D_\barZ)=0$ whenever $j\neq
m$; it is \emph{regular holonomic} if the natural restriction map from
formal to analytic solutions of $M$ is an isomorphism in the derived
category. We note that if $\mathscr{O}$ is a function space, the
space of $\mathscr{O}$-valued solutions of $M$ is
$\hom_{D_\barZ}(M,\mathscr{O})$. Thus, if $m>1$, regularity of $D_\barZ$-modules
involves the derived solutions
$\ext^j_{D_\barZ}(M,\mathscr{O})$ for $j>0$, where $\mathscr{O}$ is either the
space of formal or analytic solutions of $M$ at any given point
of $\barZ$. 

\begin{definition}
\label{def:Horn}
Let $B$ be an $n\times m$ integer matrix of full rank $m$ with rows
$B_1,B_2,\dots,B_n$, whose $\ZZ$-column span contains no nonzero vectors
with all nonnegative entries.
Let $\kappa \in \CC^n$ and $\eta \defeq [z_1\dz{1},z_2\dz{2},\dots,z_m\dz{m}]$.
Construct the following elements of $D_\barZ$:
\[
q_k \defeq \prod_{b_{ik}>0} \prod_{\ell=0}^{b_{ik}-1} ( B_i \cdot
\eta+ \kappa_i - \ell ) \quad \text{ and } \quad p_k \defeq
\prod_{b_{ik}<0} \prod_{\ell=0}^{|b_{ik}|-1} ( B_i \cdot \eta+
\kappa_i - \ell ) .
\]
\begin{enumerate}
\item \label{def:reg-Horn} The \emph{Horn hypergeometric system}
  associated to $B$ and $\kappa$ is the left $D_\barZ$-ideal
\begin{equation}
\label{eqn:horndef}
\Horn(B,\kappa) \defeq 
D_\barZ\cdot\< q_k - z_k p_k \mid k=1,2,\dots m\> 
\subseteq D_\barZ.
\end{equation}
\item \label{def:nHorn} Assume that $B$ has an $m\times m$  identity
  submatrix,
  and assume that
  the corresponding entries of $\kappa$ are all zero.  The
  \emph{normalized Horn hypergeometric system} associated to $B$ and
  $\kappa$ is the left $D_\barZ$-ideal
\begin{equation}
\label{eqn:nhorndef}
\nHorn(B,\kappa) \defeq 
D_\barZ\cdot\left\< \frac{1}{z_k} q_k - p_k \, 
\bigg\vert\, k=1,2,\dots m\right\> 
\subseteq D_\barZ.
\end{equation}
\end{enumerate}
\endrk
\end{definition}

Normalized Horn systems abound in the mathematical literature, and they 
include  the (generalized) Gauss
hypergeometric equation(s), as well as the systems of differential equations
corresponding to the Appell series, Horn series in two
variables, Lauricella series, and Kamp\'e de Feri\'et functions, among
others.
In general, Horn hypergeometric systems have proved resistant to $D$-module
theoretic study; in fact, we are aware of
only~\cite{sadykov-mathscand,DMS,bmw-functor}, which  
contain partial results regarding the holonomicity of
$\Horn(B,\kappa)$. 

In the late 1980s, Gelfand, Graev, Kapranov, and Zelevinsky introduced
a different kind of hypergeometric system, known as
\emph{$A$-hypergeometric}, or \emph{GKZ systems}, that are
much more amenable to a $D$-module theoretic approach \cite{GGZ, GKZ}. A modification of
these systems lead to \emph{lattice basis $D$-modules}, whose
solutions are in one-to-one correspondence with the solutions of Horn
systems.  

\begin{definition}
\label{def:latticeBasisDmod}
Let $B$ and $\kappa$ be as in Definition~\ref{def:Horn}, and set
$d=n-m$. Let $A=(a_{ij})$ 
be a $d\times n$ integer matrix of full rank, whose columns span
$\ZZ^d$ as a lattice, and such that $AB=0$.
Let $\barX =\CC^n$ with
coordinates $x_1,x_2,\dots,x_n$ and consider the Weyl algebra $D_\barX$
generated by the $x_i$ and their corresponding $\dx{i}=\frac{\del}{\del x_i}$. Denote
$\theta_i = x_i\dx{i}$ for $i=1,2,\dots,n$.
The polynomial ideal 
\[
I(B) \defeq \left\< \prod_{(B_j)_i>0}\dx{i}^{(B_j)_i} -
\prod_{(B_j)_i<0}\dx{i}^{(B_j)_i} \right\> \subset \CC[\dx{1},\dx{2},\dots,\dx{n}]
\]
is called a
\emph{lattice basis ideal}. Let $E_i \defeq \sum_{j=1}^n a_{ij} \theta_j$,
and denote by $E-A\kappa$ the sequence $E_1-(A\kappa)_1,E_2-(A\kappa)_2,\dots,E_d-(A\kappa)_d$, which are known as \emph{Euler operators}.
The \emph{lattice basis $D_\barX$-module associated to $B$
  and $\kappa$} is the quotient of $D_\barX$ by the left $D_\barX$-ideal
$H(B,\kappa)$ generated by $I(B)$ and $E-A\kappa$.
\endrk
\end{definition}

The solutions of Horn hypergeometric systems and lattice basis binomial $D$-modules
are related as follows. Let $B$ and $\kappa$ be as in
Definition~\ref{def:Horn}, and denote 
by $b_1,b_2,\dots,b_m$ the columns of $B$.
Let $\varphi(z)=\varphi(z_1,z_2,\dots,z_m)$ be a germ of a
holomorphic function at a generic point of $\barZ$. Then $\varphi(z)$
is a solution of 
$D_\barZ/\Horn(B,\kappa)$ if and only if $x^\kappa g(x^{b_1},x^{b_2},\dots,x^{b_m})$
is a solution of $D_\barX/H(B,\kappa)$ (at an appropriate generic point of
$\barX$). Note that this does not imply any relationship among higher
derived solutions of the corresponding modules, or about solutions at
non-generic points. 

This correspondence between the solutions of the Horn and lattice basis $D$-modules does not imply that there is a $D$-module theoretic relationship between the systems. 
This would be desirable, since lattice basis $D$-modules are fairly well understood; in particular, there are complete characterizations of their holonomicity and regularity (see Section~\ref{sec:latticeBasisDmod}), so one could hope to transfer these results from the lattice basis to the Horn setting.
Unfortunately, the following example shows that such a $D$-module theoretic relationship cannot exist in general.

\begin{example}
\label{ex:preliminaryCounterexample}
The lattice basis $D_\barX$-module corresponding to
\[
B = \left[ \begin{array}{rrr}
1  &  1 & 2 \\
-1 & -1 & 0 \\
0  &  0 & -1 \\
1  & 0  & 0 \\
0  & 1  & 0 \\
-1 & 0  &0 \\
0  &-1 & 0 
\end{array} \right]   \;\;
\text{and}\;\; \kappa= \begin{bmatrix} 2 \\ 0 \\ 0 \\ 0 \\ 0 \\ 0 \\
  0 \end{bmatrix}
\]
is holonomic, but $D_\barZ/\Horn(B,\kappa)$ is not. This can be tested explicitly using the computer algebra system
\texttt{Macaulay2} \cite{M2}.
\endrk
\end{example}

However, for \emph{normalized} Horn systems, the main result in this
article provides a relationship between these and their lattice basis
counterparts. 

\begin{theorem}
\label{thm:restriction}
Suppose that the top $m$ rows of $B$ form an identity matrix and
$\kappa_1=\kappa_2=\dots=\kappa_m=0$.  Let $r$ denote the inclusion $r\colon
\barZ \hookrightarrow \barX$ given by $(z_1,z_2,\dots,z_m) \mapsto
(z_1,z_2,\dots,z_m,1,\dots,1)$.  If $r^*$ is the restriction (inverse
image under $r$) on $D_\barX$-modules, then there is an equality
\[
\frac{D_\barZ}{\nHorn(B,\kappa)} = r^* \left(\frac{D_\barX}{
    H(B,\kappa)}\right).
\] 
\end{theorem}

This result is inspired by~\cite[\S\S11-13]{beukers-notes}. In this
work, Beukers obtains examples of classical Horn series by setting to
one certain variables in the series solutions of associated
$A$-hypergeometric systems. Theorem~\ref{thm:restriction} implies this
correspondence among series solutions, and much more.

\begin{corollary}
\label{cor:restriction} 
Under the hypotheses of Theorem~\ref{thm:restriction}, the (regular)
holonomicity of the modules $D_\barZ/\nHorn(B,\kappa)$ and $D_\barX/
H(B,\kappa)$ are equivalent.  
\end{corollary}

\subsection*{Outline}

In \S\ref{sec:restriction}, we prove
Theorem~\ref{thm:restriction}. In \S\ref{sec:latticeBasisDmod}, we recall
the characterizations for holonomicity and regularity of lattice basis
$D$-modules and prove Corollary~\ref{cor:restriction}.

\subsection*{Acknowledgements}
We are grateful to Frits Beukers, Alicia Dickenstein, Brent Doran,
Anton Leykin, Ezra Miller, Christopher O'Neill, Mikael Passare, and
Bernd Sturmfels, who have generously shared their insight and
expertise with us while we worked on this project. 
Parts of this work were carried out at the Institut Mittag-Leffler
program on Algebraic Geometry with a view towards Applications and the
MSRI program on Commutative Algebra. We thank the program organizers
and participants for exciting and inspiring research atmospheres.  

\section{Normalized Horn systems are restrictions}
\label{sec:restriction}

In this section, we prove Theorem~\ref{thm:restriction}. We use the
notation and assumptions introduced in Definitions~\ref{def:Horn} and~\ref{def:latticeBasisDmod}.

By~\cite[\S5.2]{SST}, the restriction $r^*$ of a cyclic
$D_\barX$-module $D_\barX/J$ is given by
\begin{equation}
\label{eqn:restriction}
r^*\left(\frac{D_\barX}{J}\right)=
\frac{\CC[x_1,x_2,\dots,x_n]}{\<x_{m+1}-1,x_{m+2}-1,\dots,x_n-1\>} \otimes_{\CC[x_1,x_2,\dots,x_n]}
\frac{D_\barX}{J}. 
\end{equation}
While the restriction of a cyclic $D_\barX$-module is not
necessarily cyclic, to establish Theorem~\ref{thm:restriction} first show that  $r^*(D_\barX/H(B,\kappa))=r^*(D_\barX/D_\barX\cdot\< I(B),E-A\kappa\>)$ is indeed cyclic. To do this, we
compute the $b$-function for the restriction, as defined in~\cite[\S\S5.1-5.2]{SST}.

\begin{lemma}
\label{lemma:b-function restriction}
If the matrix formed by the top $m$ rows of $B$ has rank $m$, then the
$b$-function of $H(B,\kappa)$ 
for restriction to
$\Var(x_{m+1}-1,x_{m+2}-1,\dots,x_n-1)$ divides $s$.
\end{lemma}

\begin{proof}
  Begin with the change of variables $x_j \mapsto x_j+1$ for $m+1\leq
  j \leq m$, and let $J$ denote the $D_\barX$-ideal obtained from
$H(B,\kappa)$ via this change of variables. 
We now compute the $b$-function of $J$ for
  restriction to $\Var(x_{m+1},x_{m+2},\dots,x_n)$.

With $w = (\boldzero_m,\boldone_d)\in\RR^n$, the vector $(-w,w)$
induces a filtration on $D_\barX$, and the $b$-function we wish to
compute is a generator of the principal ideal $\gr^{(-w,w)}(J) \cap \CC[s]$,
where $s \defeq \theta_{m+1}+\theta_{m+2}+\cdots+\theta_n$.  Note that, since the
submatrix of $B$ formed by its first $m$ rows has rank $m$, the
submatrix of $A$ consisting of its last $n-m=d$ columns has rank
$d$. Thus there are vectors $\nu^{(m+1)},\nu^{(m+2)},\dots,\nu^{(n)} \in \RR^d$
such that $(\nu^{(j)}A)_k = \delta_{jk}$ for $m+1\leq k\leq n$.  For
$m+1\leq j \leq n$, with $\beta=A\kappa$, 
\[
\sum_{i=1}^d \nu^{(j)}_i E_i - \nu^{(j)}\cdot\beta 
\, =  \, 
\sum_{k=1}^m (\nu^{(j)}A)_k \theta_k + \theta_j - \nu^{(j)}\cdot\beta
\, \in \, D_\barX\cdot\<E-\beta\>. 
\]
Using our change of variables and multiplying by $x_j$, for $m+1\leq j
\leq n$ we obtain
\[
\sum_{k=1}^m (\nu^{(j)}A)_k x_j\theta_k +x_j^2\dx{j} + \theta_j - \nu^{(j)}\cdot\beta x_j
\, \in \, J. 
\]
Taking initial terms with respect to $(-w,w)$ of this expression, it
follows that $\theta_j \in \gr^{(-w,w)}(J)$ for each $m+1\leq j \leq
n$. Therefore $s=\theta_{m+1}+\theta_{m+2}+\cdots+\theta_n \in \gr^{(-w,w)}(J)$,
and the result follows.
\end{proof}

\begin{proof}[Proof of Theorem~\ref{thm:restriction}]
By Lemma~\ref{lemma:b-function restriction},
$r^*(D_\barX/H(B,\kappa))$ is of the form $D_\barZ/L$. 
In order to find 
the ideal $L$, we must perform the intersection
\begin{equation}
\label{eqn:intersection}
H(B,\kappa)\cap R_m, \text{ where } R_m 
\defeq \CC[x_1,x_2,\dots,x_n]\<\dx{1},\dx{2},\dots,\dx{m}\> \subseteq D_\barX, 
\end{equation}
and then set $x_{m+1}=x_{m+1}=\cdots=x_n=1$.  We proceed by systematically
producing elements of the intersection~\eqref{eqn:intersection}.
Using the same argument as in the proof of Lemma~\ref{lemma:b-function
  restriction}, we see that for $m+1 \leq j \leq n$, each $\theta_j$
can be expressed modulo $D_\barX\cdot\<E-\beta\>$ as a linear
combination of $\theta_1,\theta_2,\dots,\theta_m$ and the parameters $\kappa$.
By our assumption on $B$, $\theta_j$ can be written explicitly as
follows:
\begin{equation}
\label{eqn:substitution}
\theta_j = \kappa_j + \sum_{i=1}^m b_{ji}
  \theta_i \quad \textrm{mod}\; D_\barX\cdot\<E-\beta\> \qquad
  \text{for }\, m+1 \leq j \leq n. 
\end{equation}

Now if $P \in D_\barX$, then there is a monomial $\mu$ in
$x_{m+1},x_{m+2},\dots,x_n$ so that the resulting operator $\mu P$ can be
written in terms of $x_1,x_2,\dots,x_n,\dx{1},\dx{2},\dots,\dx{m}$, and
$\theta_{m+1},\theta_{m+2},\dots,\theta_n$. 
In addition, working modulo
$D_\barX\cdot\<E-\beta\>$, one can replace $\theta_j$ when $j>m$ by
the expressions~\eqref{eqn:substitution}.  Thus $\mu P$ is an element
of $R_m$ modulo $R_m\cdot\<E-\beta\>$. 
If this procedure is applied
to $E_i-\beta_i$, the result is zero.  We now apply it to one of
the generators $\del_x^{(b_k)_+}-\del_x^{(b_k)_-}$ of $I(B)$, where
$b_1,b_2,\dots,b_m$ denote the columns of $B$. An appropriate monomial in
this case is $\mu_k =\prod_{j=m+1}^{n} x_j^{|b_{jk}|}$.  Then the fact
that $b_{kk}=1$ for $1\leq k \leq m$ and~\eqref{eqn:substitution}
together imply that
\begin{align}
\mu_k( & \del_x^{(b_k)_+} -\del_x^{(b_k)_-}) 
\nonumber
\\
\nonumber
&= \big(\textstyle\prod_{b_{jk}<0} x_j^{-b_{jk}}\big) \dx{k}^{b_{kk}}
\textstyle\prod_{j>m, b_{jk}>0 } x_j^{b_{jk}}\dx{j}^{b_{jk}} - 
  \textstyle\prod_{j>m, b_{jk}>0} x_j^{b_{jk}}
  \textstyle\prod_{b_{jk}<0} x_j^{-b_{jk}}\dx{j}^{-b_{jk}} 
\nonumber
\\
\begin{split}
&= \big(\textstyle\prod_{b_{jk}<0} x_j^{-b_{jk}}\big) \dx{k} \textstyle\prod_{j>m,
  b_{jk}>0} \textstyle\prod_{\ell=0}^{b_{jk}-1} (\kappa_j + \sum_{i=1}^m
  b_{ji} \theta_i - \ell) 
\\ & \qquad\qquad\qquad\qquad\qquad
 -
  \textstyle\prod_{j>m, b_{jk}>0} x_j^{b_{jk}} \textstyle\prod_{b_{jk}<0}
  \textstyle\prod_{\ell=0}^{-b_{jk}-1} (\kappa_j + \sum_{i=1}^m
   b_{ji} \theta_i-\ell). 
   \label{eq:almost:nHorn}
\end{split}
\end{align}
Note that setting 
$x_{m+1}=x_{m+2}=\cdots=x_n=1$ in~\eqref{eq:almost:nHorn}, we obtain the $k$th
generator of the normalized Horn system $\nHorn(B,\kappa)$, since
$b_{jk}<0$ implies $j>m$. 
This shows that $\nHorn(B,\kappa)$ is contained in the intersection \eqref{eqn:intersection} after setting $x_{m+1}=x_{m+2}=\cdots=x_n=1$.

Now suppose that $P$ is an element of the
intersection~\eqref{eqn:intersection}. In particular, $P$ belongs to
$I(B)+\<E-\beta\>$, so there are $P_1,P_2,\dots,P_m,Q_1,Q_2,\dots,Q_d \in
D_\barX$ such that  
\[
P = \sum_{k=1}^m P_k (\del_x^{(b_k)_+}-\del_x^{(b_k)_-}) +
\sum_{i=1}^d Q_i (E_i-\beta_i).
\] 
If we multiply $P$ on the left by a monomial in $x_{m+1},x_{m+2},\dots,x_n$
and set $x_{m+1}=x_{m+2}=\cdots=x_n=1$, the result is the same as if we
set $x_{m+1}=x_{m+2}=\cdots=x_n=1$ on $P$ directly. Thus we choose an
appropriate monomial $\mu$ such that
\[
\mu P = \sum_{k=1}^m \tilde{P}_k \mu_k (\del_x^{(b_k)_+}-\del_x^{(b_k)_-}) +
\mu \sum_{i=1}^d Q_i (E_i-\beta_i)
\]
for some operators $\tilde{P}_1,\tilde{P}_2,\dots,\tilde{P}_m$. But then, the
result of setting $x_{m+1}=x_{m+2}=\cdots=x_n=1$ on $\mu P$ (the same
as if this were done to $P$) is a combination of the generators of
$\nHorn(B,\kappa)$. 
Thus, we have shown that the intersection \eqref{eqn:intersection}
after setting $x_{m+1}=x_{m+2}=\cdots=x_n=1$ is contained in
$\nHorn(B,\kappa)$. 
We conclude that $r^*(D_\barZ/H(B,\kappa))=D_\barZ/\nHorn(B,\kappa)$.
\end{proof}

\section{Lattice basis \texorpdfstring{$D$}{D}-modules}
\label{sec:latticeBasisDmod}

The ring $D_\barX$ is $\ZZ^d$-graded by setting
$\deg(\del_{x_i})=-\deg(x_i)=a_i$, where $a_1,\dots,a_n$ are the
columns of the matrix $A$ from Definition~\ref{def:latticeBasisDmod}.
This grading, which is also inherited by the
polynomial ring $\CC[\del_x]\defeq\CC[\dx{1},\dx{2},\dots,\dx{n}]$, is known as the \emph{$A$-grading}. 
An \emph{$A$-graded binomial $\CC[\del_x]$-ideal} $I$ is an ideal generated by elements of the form 
$\del_x^u - \lambda \del_x^v$. 
(In this definition, $\lambda = 0$ is allowed; in other words,
monomials are admissible generators in a binomial ideal.) 

Note that $H(B,\kappa)$ is $A$-graded, and then lattice basis binomial $D_\barX$-modules are $A$-graded. It is this
grading that can be used to determine the set of parameters $\kappa$ for which
$D_\barX/H(B,\kappa)$-module is holonomic
(Theorem~\ref{thm:toralHolonomic}). We need the notion of quasidegrees
of a module, originally introduced in~\cite{MMW}.

\begin{definition}
\label{def:quasidegrees}
Let $M$ be an $A$-graded $\CC[\del_x]$-module. The set of \emph{true
  degrees of $M$} is 
\[
\tdeg(M) = \{ \beta \in \CC^d \mid M_\beta \neq 0 \}.
\]
The set of \emph{quasidegrees of $M$}, denoted $\qdeg(M)$, is the
Zariski closure in $\CC^d$ of $\tdeg(M)$.
\endrk
\end{definition}

\begin{definition}[{\cite[Definition~4.3]{DMM/primdec},~\cite[Definitions~1.11
    and~6.9]{DMM/D-mod}}]
\label{def:toralAndean}
Let $A$ be as in Definition~\ref{def:latticeBasisDmod}, and let $I$ be
an $A$-graded 
binomial $\CC[\del_x]$-ideal. By~\cite{eisenbud-sturmfels}, any
associated prime of 
$I$ is of the form $\CC[\del_x]  \cdot J+\<x_j \mid j \notin
\sigma\>$, where $\sigma 
\subset \{1,2,\dots,n\}$ and $J \subset \CC[\del_{x_i} \mid i \in
\sigma]$ is a prime binomial ideal containing no monomials. Such an
associated prime is called 
\emph{toral} if the dimension of $\CC[\del_{x_i} \mid i \in \sigma]/J$
equals the 
rank of the submatrix of $A$ consisting of the columns indexed by
$\sigma$. An associated prime of $I$ which is not toral is called
\emph{Andean}. 

Consider a primary decomposition $I = \bigcap_{\ell = 1}^N C_\ell$, where
$C_1,C_2,\dots,C_K$ are the primary components corresponding to Andean
associated primes and $C_{K+1},C_{K+2},\dots,C_N$ are the components
corresponding to toral associated primes. The \emph{Andean
  arrangement of $I$} is 
\[
\hspace*{145.5pt}
\mathscr{Z}_{\textnormal{Andean}}(I)\defeq
\bigcup_{\ell=1}^K \qdeg \left(\CC[\del_x]/C_\ell\right).
\hspace*{145.5pt}
\hexagon
\] 
\end{definition}

The name \emph{Andean} refers to an intuitive
picture of the grading of an Andean module
(see~\cite[Remark~5.3]{DMM/D-mod}). 

Since Andean primes may be embedded, the definition of the Andean
arrangement seems a priori to depend on the specific primary
decomposition; however,~\cite[Theorem~6.3]{DMM/D-mod} shows that this is not the case. We will make use of the following 
Theorem~\ref{thm:toralHolonomic}, whose first part is a special case
of~\cite[Theorem~6.3]{DMM/D-mod}, while its second part is proved in~\cite{binomial-slopes}.  

We recall that the \emph{holonomic rank} of a $D$-module is the
dimension of its space of germs of holomorphic solutions at a
generic (nonsingular) point.

\begin{theorem}
\label{thm:toralHolonomic}
Use the notation from Definitions~\ref{def:latticeBasisDmod}
and~\ref{def:toralAndean}. The following 
are equivalent.
\vspace{-2mm}
\begin{enumerate}
\item The $D_\barX$-module $D_\barX/H(B,\kappa)$ has
  finite holonomic rank.
\item The $D_\barX$-module $D_\barX/H(B,\kappa)$ is
  holonomic.
\item $A\kappa \notin \mathscr{Z}_{\textnormal{Andean}}(H(B,\kappa))$.
\end{enumerate}
\vspace{-2mm}
In addition, $D_\barX/H(B,\kappa)$ is regular holonomic if and only if
it is holonomic and the rows of $B$ sum to $\mathbf{0}_m$.
\qed
\end{theorem}

We need one more result in order to prove Corollary~\ref{cor:restriction}.
Let $Z=(\CC^*)^n$, and consider its ring of differential operators $D_Z \defeq \CC[z_1^{\pm 1},\dots,z_m^{\pm
  1}]\otimes_{\CC[z_1,z_2,\dots,z_m]} D_\barZ$.
The \emph{saturated Horn system corresponding to $B$ and $\kappa$} is 
$\sHorn(B,\kappa) \defeq D_Z \cdot \Horn(B,\kappa) \cap D_\barZ$.

\begin{theorem}[{\cite[Corollary~7.2]{bmw-functor}}]
\label{thm:satHorn}
The $D_\barX$-module $D_\barX/H(B,\kappa)$ is (regular) holonomic if and
only if the $D_\barZ$-module $D_\barZ/\sHorn(B,\kappa)$ is (regular) holonomic.
\qed
\end{theorem}

\begin{proof}[Proof of Corollary~\ref{cor:restriction}]
  If $D_\barX/H(B,\kappa)$ is (regular) holonomic, then so
  is $D_\barZ/\nHorn(B,\kappa)$ by Theorem~\ref{thm:restriction},
  since restrictions preserve (regular) holonomicity.  
  For the
  converse, if $D_\barX/H(B,\kappa)$ is not (regular)
  holonomic, then neither is $D_\barZ/\sHorn(B,\kappa)$ by
  Theorem~\ref{thm:satHorn}. Since
  $\nHorn(B,\kappa) \subseteq \sHorn(B,\kappa)$, and the category of
  (regular) holonomic $D_\barZ$-modules is closed under quotients of
  $D_\barZ$-modules, $D_\barZ/\nHorn(B,\kappa)$ also fails to be
  (regular) holonomic.
\end{proof}

\raggedbottom
\def\cprime{$'$} \def\cprime{$'$}
\providecommand{\href}[2]{#2}
\end{document}